\documentclass[a4paper,12pt]{article}
\usepackage{a4}
\usepackage{amsmath}
\usepackage{color}
\usepackage{amsfonts}
\usepackage{mathrsfs}
\usepackage{amssymb}
\usepackage{mathtools}

\usepackage[square, numbers, comma, sort&compress]{natbib}
\usepackage{amsthm,enumerate}
\usepackage{graphicx}
\usepackage[colorlinks=true,citecolor=black,linkcolor=black,urlcolor=blue]{hyperref}

\def\ni{\noindent}
\makeatletter
\newtheorem{theorem}{Theorem}[section]

\newtheorem{lemma}[theorem]{Lemma}

\newtheorem{cor}[theorem]{Corollary}

\def\ni{\noindent}
\numberwithin{equation}{section}

\begin{document}
	
	\begin{center}
		{\Large \bf Reciprocal distance energy of complete\\[2mm] multipartite graphs}\\
		
		\vspace{10mm}
		
		{\large \bf Rakshith B. R.}$^{1,\ast}$,  {\large \bf B. J. Manjunatha}$^{2,3}$
		
		\vspace{9mm}

		\baselineskip=0.20in

		$^1${\it Department of Mathematics, Manipal Institute of Technology,\\ Manipal Academy of Higher Education,\\
			Manipal 576 104, India\/} \\
		{\rm E-mail:} {\tt ranmsc08@yahoo.co.in}\\[2mm]
		
		$^2${\it Research Scholar, Department of Mathematics\\ Vidyavardhaka College of Engineering\\
			Mysuru-570 002, India\/}\\[2mm] 
		
		$^3${\it Department of Mathematics, Sri Jayachamarajendra College of Engineering,\\ JSS Science and Technology University,\\
			Mysuru–570 006, India\/} \\
		{\rm E-mail:} {\tt manjubj@sjce.ac.in}
		
		\vspace{4mm}
		
	\end{center}
	
	\vspace{5mm}
	
	\baselineskip=0.20in
	
	\begin{abstract}
		\ni In this paper, first we compute the energy of a special partitioned matrix under some cases. As a consequence, we obtain the reciprocal distance energy of the complete multipartite graph and also we give various other energies of complete multipartite graphs.
		Next, we show that among all complete $k$-partite graphs on $n$ vertices, the complete split graph $CS(n,k-1)$ has minimum reciprocal distance energy and the reciprocal distance energy is maximum for the Turan graph $T(n,k)$. At last, it is shown that the reciprocal distance energy of the complete bipartite graph $K_{m,m}$ decreases under deletion of an edge if $2\le m\le 7$, whereas the reciprocal distance energy increases if $8\le m$. Also, we show that the reciprocal distance energy of the complete tripartite graph does not increase under edge deletion. 
		
	\end{abstract}
	\ni  \textbf{Keywords}: Complete multipartite graphs, reciprocal distance matrix, reciprocal distance energy. \\\\
	\ni \textbf{MSC (2010)}:  05C50.
	\baselineskip=0.30in
	\section{Introduction}
		Graphs considered in this paper are simple, connected and undirected. We denote the eigenvalues of a Hermitian matrix $H$ of order $n$ by $\lambda_{1}(H)\ge\lambda_{2}(H)\ge\ldots\ge\lambda_{n}(H)$. For a graph $G$ with vertex set $V(G)=\{v_{1},v_{2},\ldots,v_{n}\}$ and edge set $E(G)$, the distance between two distinct vertices $v_{i}$ and $v_{j}$ is denoted by $d(v_{i},v_{j})$ and is equal to the length of a shortest path connecting the vertices $v_{i}$ and $v_{j}$. The reciprocal distance  matrix of $G$, well-known as Harary matrix, is a symmetric matrix of order $n$, denoted by $RD(G)$ and its \textit{ij}th entry is equal to $\dfrac{1}{d(v_{i},v_{j})}$ if $i\neq j$, 0 otherwise. This matrix was introduced by Ivanciuc et al. for the design of topological indices in the year 1993, see \cite{ivanciu}. The well-known topological index derived from the reciprocal distance matrix is the Harary index, see \cite{hbook}. In \cite{ivan}, Ivanciuc et al. used the largest eigenvalue of the $RD(G)$ as one of the structural descriptor  to develop structure–property models for the normal boiling temperature, molar heat capacity, standard Gibbs energy of
		formation, vaporization enthalpy, refractive index, and density of 134 alkanes $C_6$–$C_{10}$. In \cite{dasm}, Das  obtained a lower and upper bound for the largest
		eigenvalue of the reciprocal distance matrix of a graph. Also, Nordhaus-Gaddum-type bounds for the largest eigenvalue of the reciprocal distance matrix were obtained therein. Graphs with maximum spectral radius of the reciprocal distance matrix in the classes of  graphs (bipartite graphs) with fixed matching number and  graphs with given number of cut edges were determined in \cite{amc}.  \\[2mm] Energy of a graph is a well-known graph invariant derived from the adjacency spectrum of a graph. This graph invariant nowadays known as ordinary energy of a graph was  introduced by Gutman in connection with H$\ddot{\text{u}}$ckel theory \cite{gut}. In analogous to the definition of graph energy, G$\ddot{\text{u}}$ng$\ddot{\text{o}}$r and Cevik in \cite{gungor} introduced Harary energy of a graph, also called reciprocal distance energy of a graph. It is denoted by $\mathcal{E}_{RD}(G)$ and is defined as $\mathcal{E}_{RD}(G)=\sum_{i=1}^{n}|\lambda_{i}(RD(G))|$. Several lower and upper bounds for the reciprocal distance energy in terms of graph parameters are given in \cite{rde1,rde2,rde3,akbar}. In \cite{eq1}, Ramane et al. constructed pairs of reciprocal distance equienergetic graphs by determining the reciprocal distance energy of  line graph of  certain regular graph,  and its complement. In \cite{eq2}, reciprocal distance equienergetic graphs are presented using the reciprocal distance spectrum of some generalized composition of graphs. Recent studies on the reciprocal distance matrix can be found in \cite{rds1,laa}.\\[2mm]
		 The energy of a complex matrix $M$ is the sum of all singular values of the matrix $M$ and is denoted by $\mathcal{E}(M)$. The definition of energy of a complex matrix was put forward by Nikiforov as an extension of the concept of graph energy, see \cite{nikiforov} for more details.
		We denote a complete $k$-partite graph by $K_{n_{1},n_{2},\ldots,n_{k}}$. The complete split graph $CS(n,k)$ is a graph on $n$ vertices obtained by taking one copy of the complete graph $K_{k}$ and joining each of its vertices with $n-k$ isolated vertices, i.e., $CS(n,k)\cong K_{n-k,1,1,\ldots,1}$. The Turan graph $T(n,k)$ is the complete $k$-partite graph on $n$ vertices given by $T(n,k)\cong K_{q+1,q+1,\ldots,q+1,q,q\ldots,q}$, where $n=kq+r$ and $r\ge0$. We denote the adjacency matrix of $G$ by $A(G)$.  For terminologies not defined in the paper, we refer to \cite{energybook,e2019}. \\[2mm]
		 In Section 2 of the paper, we compute the energy of a special partitioned matrix under some cases. As a consequence, we obtain the reciprocal distance energy of the complete multipartite graph and also we give various other energies of complete multipartite graphs.
		 In Section 3, we show that among all complete $k$-partite graph on $n$ vertices, the complete split graph $CS(n,k-1)$ has minimum reciprocal distance energy and the reciprocal distance energy is maximum for the Turan graph $T(n,k)$. In Section 4, it is shown that the reciprocal distance energy of the complete bipartite graph $K_{m,m}$ decreases under deletion of an edge if $2\le m\le 7$, whereas the reciprocal distance energy increases if $8\le m$. Also, we show that the reciprocal distance energy of the complete tripartite graph does not increase under edge deletion. 
	\section{Energy of a special partitioned matrix}
	Let $M_{1}, M_{2},\ldots, M_{k}$ ($k\ge2$) be real symmetric matrices of order $n_{1}, n_{2}, \ldots, n_{k}$ such that $M_{i}\textbf{1}_{n_{i}}=r_{i}\textbf{1}_{n_{i}}$, where $r_{i}\ge 0$, and $trace(M_{i})=0$. Denote by  $M=M[M_{1}, M_{2},\ldots,M_{k},a]$, a square matrix of order $n=n_{1}+n_{2}+\ldots+n_{k}$, defined as
	$$M=\left[\begin{array}{ccccc}
		M_{1} &aJ_{n_{1}\times n_{2}}&aJ_{n_{1}\times n_{3}}&\ldots& aJ_{n_{1}\times n_{k}}\\[2mm]	
		aJ_{n_{2}\times n_{1}}&M_{2}&aJ_{n_{2}\times n_{3}}&\ldots& aJ_{n_2\times n_{k}}\\[2mm]	
		\vdots&\vdots&\ddots&\vdots&\vdots\\[2mm]	
		aJ_{n_{k-1}\times n_{1}}&aJ_{n_{k-1}\times n_{2}}&\ldots& M_{k-1}&aJ_{n_{k-1}\times n_{k}}\\[2mm]	
		aJ_{n_{k}\times n_{1}}&aJ_{n_{k}\times n_{2}}&\ldots&aJ_{n_{k}\times n_{k-1}}& M_{k}\\
	\end{array}\right],$$
	where $a$ ($\neq 0$) is a real constant and $J_{n_{i}\times n_{j}}$ is a rectangular matrix of order $n_{i}\times n_{j}$ with all its entries equal to 1.\\[2mm]  
	Let $\dfrac{r_{1}}{n_{1}}=\displaystyle \max_{1\le i\le k}\left\{\dfrac{r_{i}}{n_{i}}\right\}$ and $\dfrac{r_{k}}{n_{k}}=\displaystyle\min_{1\le i\le k}\left\{\dfrac{r_{i}}{n_{i}}\right\}$.  In this section, we compute the energy of the matrix $M=M[M_{1}, M_{2}, \ldots, M_{k},a]$ under some cases. As a consequence, we determine the reciprocal distance energy of complete multipartite graph and also other energies of complete multipartite graph are obtained. For a Hermitian matrix $H$, $S^{-}(H)$ denote the sum of all negative eigenvalues of $H$.\\[2mm] 
	The following lemma is a quantitative formulation of
	Sylvester's law of inertia due to Ostrowski.
	\begin{lemma}[Ostrowski \cite{ostrowski}]\label{ost}
	Let $A$ be a Hermitian matrix of order $n$ and $S$ be a real non-singular matrix of order $n$. Then $\lambda_{i}(S^{T}AS)=\theta_{i}\lambda_{i}(A)$, where 	
$\lambda_{n}(S^{T}S)\le \theta_{i}\le\lambda_{1}(S^{T}S)$.	
\end{lemma}
\ni	The following  result is widely used in the study of graph eigenvalues.
	\begin{lemma}{\rm\cite{book}}\label{l2}
		Let $M=N+P$, where N and P are Hermitian matrices of order n. Then for $1\le i, j \le n$, we have (i) $\lambda_{i}(N)+\lambda_{j}(P)\le \lambda_{i+j-n}(M)~(i+j>n)$ and (ii) 
		$\lambda_{i+j-1}(M)\le \lambda_{i}(N)+\lambda_{j}(P)~(i+j-1\le n)$.
	\end{lemma}
\begin{theorem}\label{thm1}
For the matrix $M=M[M_{1}, M_{2}, \ldots, M_{k},a]$ as defined above. We have\\[2mm]
i. $\mathcal{E}(M)=\displaystyle\sum_{i=1}^{k}\mathcal{E}(M_{i})$ if $a>0$ and $-a+\dfrac{r_{k}}{n_{k}}\ge 0.$	\\[2mm]
ii. $\mathcal{E}(M)=2\lambda_{1}(M)$ if $M_{i}$ has at most one positive eigenvalue, namely $r_{i}$, for $1\le i\le k$, and $-a+\dfrac{r_{1}}{n_{1}}\le 0$ .\\[2mm]
iii. $\mathcal{E}(M)=\displaystyle\sum_{i=1}^{k}\mathcal{E}(M_{i})-2\lambda_{n}(M)$ if $a<0$ and $a(k-1)+\dfrac{r_{k}}{n_{k}}\le \lambda_{n_{i}}\le 0$ for all $1\le i\le k$.
\end{theorem}
\begin{proof}
Let $s_{i}$ denote the number of positive eigenvalues of the matrix $M_{i}$ and	
let $\lambda_{i1}\ge \lambda_{i2}\ge \ldots\ge \lambda_{is_{i}}\ge\lambda_{i(s_{i}+1)}\ge \ldots\ge \lambda_{in_{i}}$ be the eigenvalues of the matrix $M_{i}$. Then \begin{eqnarray}\label{eq1}\mathcal{E}(M_{i})&=&\lambda_{i1}+\lambda_{i2}+\ldots+\lambda_{i(s_{i}-1)}+\lambda_{is_{i}}-\lambda_{i(s_{i}+1)}-\ldots-\lambda_{in_{i}}\notag\\
&=&-2(\lambda_{i(s_{i}+1)}+\lambda_{i((s_{i}+2))}+\ldots+\lambda_{in_{i}})~~(\text{ Because}, trace(M_{i})=0)\notag\\
&=&-2S^{-}(M_{i})
\end{eqnarray} Since $M_{i}$ is a real symmetric matrix of $n_{i}$, there exists an orthogonal matrix $P_{i}$ such that $P^{T}_{i}M_{i}P_{i}=D_{i}$, where $D_{i}$ is a diagonal matrix of order $n_{i}$ having $\lambda_{i1}, \lambda_{i2}, \ldots, \lambda_{in_{i}}$ as its diagonal entries. Further, since $M_{i}1_{n_{i}}=r_{i}1_{n_{i}}$, we have $\lambda_{i\ell}=r_{i}$ for some $1\le \ell\le n_{i}$, and without loss of generality, we can assume that the first column of $P_{i}$ is equal to $\dfrac{1_{n_{i}}}{\sqrt{n_{i}}}$. Therefore, $P^{T}_{i}J_{n_{i}\times n_{j}}P_{j}=\sqrt{n_{i}n_{j}}e_{n_{i}\times n_j}$, where $e_{n_{i}\times n_{j}}$ is a rectangular matrix of order $n_{i}\times n_{j}$ whose all entries are 0, except the first diagonal entry which is equal to  1. Consider
\begin{align*}
M&=\left[\begin{array}{ccccc}
	M_{1} &aJ_{n_{1}\times n_{2}}&aJ_{n_{1}\times n_{3}}&\ldots& aJ_{n_{1}\times n_{k}}\\[2mm]	
	aJ_{n_{2}\times n_{1}}&M_{2}&aJ_{n_{2}\times n_{3}}&\ldots& aJ_{n_2\times n_{k}}\\[2mm]	
	\vdots&\vdots&\ddots&\vdots&\vdots\\[2mm]	
	aJ_{n_{k-1}\times n_{1}}&aJ_{n_{k-1}\times n_{2}}&\ldots& M_{k-1}&aJ_{n_{k-1}\times n_{k}}\\[2mm]	
	aJ_{n_{k}\times n_{1}}&aJ_{n_{k}\times n_{2}}&\ldots&aJ_{n_{k}\times n_{k-1}}& M_{k}\\
\end{array}\right]\end{align*}\begin{align*}
&=\left[\begin{array}{ccccc}
	P_{1}D_{1}P^{T}_{1} &aJ_{n_{1}\times n_{2}}&aJ_{n_{1}\times n_{3}}&\ldots& aJ_{n_{1}\times n_{k}}\\[2mm]	
	aJ_{n_{2}\times n_{1}}&P_{2}D_{2}P_{2}^{T}&aJ_{n_{2}\times n_{3}}&\ldots& aJ_{n_2\times n_{k}}\\[2mm]	
	\vdots&\vdots&\ddots&\vdots&\vdots\\[2mm]	
	aJ_{n_{k-1}\times n_{1}}&aJ_{n_{k-1}\times n_{2}}&\ldots& P_{k-1}D_{k-1}P^{T}_{k-1}&aJ_{n_{k-1}\times n_{k}}\\[2mm]	
	aJ_{n_{k}\times n_{1}}&aJ_{n_{k}\times n_{2}}&\ldots&aJ_{n_{k}\times n_{k-1}}& P_{k}D_{k}P^{T}_{k}\\
\end{array}\right]\\[10mm]
&=\left[\begin{array}{ccccc}
	P_{1} &0&0&\ldots& 0\\[2mm]	
	0&	P_{2}&0&\ldots& 0\\[2mm]	
	\vdots&\vdots&\ddots&\vdots&\vdots\\[2mm]	
	0&0&\ldots& P_{k-1}&0\\[2mm]	
	0&0&\ldots&0&P_{k}\\
\end{array}\right]\left[\begin{array}{ccccc}
D_{1} &aP^{T}_{1}J_{n_{1}\times n_{2}}P_{2}&aP^{T}_{1}J_{n_{1}\times n_{3}}P_{3}\\[2mm]	
aP^{T}_{2}J_{n_{2}\times n_{1}}P_{1}&D_{2}&aP^{T}_{2}J_{n_{2}\times n_{3}}P_{3}\\[2mm]	
\vdots&\vdots&\ddots\\[2mm]	
aP^{T}_{k-1}J_{n_{k-1}\times n_{1}}P_{1}&aP^{T}_{k-1}J_{n_{k-1}\times n_{2}}P_{2}&\ldots\\[2mm]	
aP^{T}_{k}J_{n_{k}\times n_{1}}P_{1}&aP^{T}_{k}J_{n_{k}\times n_{2}}P_{2}&\ldots\\
\end{array}\right.\\[2mm]\\&
\left.\begin{array}{ccccc}
	\ldots& aP^{T}_{1}J_{n_{1}\times n_{k}}P_{k}\\[2mm]	
	\ldots& aP^{T}_{2}J_{n_2\times n_{k}}P_{k}\\[2mm]	
	\vdots&\vdots\\[2mm]	
	 D_{k-1}&aP^{T}_{k-1}J_{n_{k-1}\times n_{k}}P_{k}\\[2mm]	
	aP^{T}_{k}J_{n_{k}\times n_{k-1}}P_{k-1}& D_{k}\\
\end{array}\right]\left[\begin{array}{ccccc}
	P^{T}_{1} &0&0&\ldots& 0\\[2mm]	
	0&	P^{T}_{2}&0&\ldots& 0\\[2mm]	
	\vdots&\vdots&\ddots&\vdots&\vdots\\[2mm]	
	0&0&\ldots& P^{T}_{k-1}&0\\[2mm]	
	0&0&\ldots&0&P^{T}_{k}\\
\end{array}\right]\\[10mm]
\end{align*}
\begin{eqnarray*}\label{eq2}
&=&\left[\begin{array}{ccccc}
	P_{1} &0&0&\ldots& 0\\[2mm]	
	0&	P_{2}&0&\ldots& 0\\[2mm]	
	\vdots&\vdots&\ddots&\vdots&\vdots\\[2mm]	
	0&0&\ldots& P_{k-1}&0\\[2mm]	
	0&0&\ldots&0&P_{k}
\end{array}\right]\left[\begin{array}{ccccc}
D_{1} &a\sqrt{n_{1}n_{2}}e_{n_1\times n_2}&a\sqrt{n_{1}n_{3}}e_{n_1\times n_3}\\[2mm]	
a\sqrt{n_{2}n_{1}}e_{n_2\times n_1}&D_{2}&a\sqrt{n_{2}n_{3}}e_{n_2\times n_3}\\[2mm]	
\vdots&\vdots&\ddots\\[2mm]	
a\sqrt{n_{k-1}n_{1}}e_{n_{k-1}\times n_1}&a\sqrt{n_{k-1}n_{2}}e_{n_{k-1}\times n_2}&\ldots\\[2mm]	
a\sqrt{n_{k}n_{1}}e_{n_k\times n_1}&a\sqrt{n_{k}n_{2}}e_{n_k\times n_2}&\ldots\\
\end{array}\right.\notag\\[10mm]\end{eqnarray*}
\begin{eqnarray*}
&&
\left.\begin{array}{ccccc}
	\ldots& a\sqrt{n_{1}n_{k}}e_{n_1\times n_k}\\[2mm]	
	\ldots& a\sqrt{n_{2}n_{k}}e_{n_2\times n_k}\\[2mm]	
	\vdots&\vdots\\[2mm]	
	 D_{k-1}&a\sqrt{n_{k-1}n_{k}}e_{n_{k-1}\times n_k}\\[2mm]	
	a\sqrt{n_{k}n_{k-1}}e_{n_k\times n_{k-1}}& D_{k}\\
\end{array}\right]
\left[\begin{array}{ccccc}
	P^{T}_{1} &0&0&\ldots& 0\\[2mm]	
	0&	P^{T}_{2}&0&\ldots& 0\\[2mm]	
	\vdots&\vdots&\ddots&\vdots&\vdots\\[2mm]	
	0&0&\ldots& P^{T}_{k-1}&0\\[2mm]	
	0&0&\ldots&0&P^{T}_{k}\\
\end{array}\right].\notag\\[2mm]
\end{eqnarray*}
Let\begin{align*}M^{\prime}=&\left[\begin{array}{ccccc}
	D_{1} &a\sqrt{n_{1}n_{2}}e_{n_1\times n_2}&a\sqrt{n_{1}n_{3}}e_{n_1\times n_3}\\[2mm]	
	a\sqrt{n_{2}n_{1}}e_{n_2\times n_1}&D_{2}&a\sqrt{n_{2}n_{3}}e_{n_2\times n_3}\\[2mm]	
	\vdots&\vdots&\ddots\\[2mm]	
	a\sqrt{n_{k-1}n_{1}}e_{n_{k-1}\times n_1}&a\sqrt{n_{k-1}n_{2}}e_{n_{k-1}\times n_2}&\ldots\\[2mm]	
	a\sqrt{n_{k}n_{1}}e_{n_k\times n_1}&a\sqrt{n_{k}n_{2}}e_{n_k\times n_2}&\ldots\\
\end{array}\right.\\[10mm]
&\left.\begin{array}{ccccc}
	\ldots& a\sqrt{n_{1}n_{k}}e_{n_1\times n_k}\\[2mm]	
	\ldots& a\sqrt{n_{2}n_{k}}e_{n_2\times n_k}\\[2mm]	
	\vdots&\vdots\\[2mm]	
	 D_{k-1}&a\sqrt{n_{k-1}n_{k}}e_{n_{k-1}\times n_k}\\[2mm]	
	a\sqrt{n_{k}n_{k-1}}e_{n_k\times n_{k-1}}& D_{k}\\
\end{array}\right].\notag    	
\end{align*}
Then from equation (\ref{eq2}) the matrices $M$ and $M^{\prime}$ are similar. Thus $Spec(M)=Spec(M^{\prime})$. 
Now, consider $det(xI-M^{\prime})$. Expanding $det(xI-M^{\prime})$ by Laplace's method along all the columns except 1$^{\text{st}}$, $(n_{1}+1)^{\text{th}}$, $(n_{1}+n_{2}+1)^{\text{th}},\ldots,(n_{1}+n_{2}+\ldots+n_{k-1}+1)^{\text{th}}$ columns, we get
$$det(xI-M^{\prime})=det(xI-M^{\prime\prime})\displaystyle\prod_{i=1}^{k} \big[det(xI-D_{i})(x-r_{i})^{-1}\big], $$
where \begin{align*}M^{\prime\prime}=&\left[\begin{array}{ccccc}
		r_{1} &a\sqrt{n_{1}n_{2}}&a\sqrt{n_{1}n_{3}}&\ldots&a\sqrt{n_{1}n_{k}}\\[2mm]	
		a\sqrt{n_{2}n_{1}}&r_{2}&a\sqrt{n_{2}n_{3}}&\ldots&a\sqrt{n_{2}n_{k}}\\[2mm]	
		\vdots&\vdots&\ddots&\vdots&\vdots\\[2mm]	
		a\sqrt{n_{k-1}n_{1}}&a\sqrt{n_{k-1}n_{2}}&\ldots&r_{k-1}&a\sqrt{n_{k-1}n_{k}}\\[2mm]	
		a\sqrt{n_{k}n_{1}}&a\sqrt{n_{k}n_{2}}&\ldots&a\sqrt{n_{k}n_{k-1}}&r_{k}\\
	\end{array}\right].
\end{align*}
Thus, \begin{equation}\label{spec}
	Spec(M)=\displaystyle\bigcup_{i=1}^{k} (Spec(M_{i})\backslash\{r_{i}\})\bigcup Spec(M^{\prime\prime}).\end{equation} Since $trace(M)=0$ and $r_{i}\ge 0$, we get 
\begin{eqnarray}\label{eq3}
\mathcal{E}(M)&=&-2\sum_{i=1}^{k}S^{-}(M_{i})-2S^{-}(M^{\prime\prime})\notag\\&=&\sum_{i=1}^{k}\mathcal{E}(M_{i})-2S^{-}(M^{\prime\prime}), \text{by equation}~(\ref{eq1}). 	
\end{eqnarray}
Note that $M^{\prime\prime}=C(aJ_{k\times k}-aI_{k}+D)C$, where $C=diag(\sqrt{n_{1}},\sqrt{n_{2}},\ldots,\sqrt{n_{k}})$ and $D=diag\left(\dfrac{r_{1}}{n_{1}},\dfrac{r_{2}}{n_{2}},\ldots,\dfrac{r_{k}}{n_{k}}\right).$ Thus the matrices $M^{\prime\prime}$ and $(aJ_{k\times k}-aI_{k})+D$ are congruent to each other. Thus by Sylvester's law of inertia the matrices $M^{\prime\prime}$ and $(aJ_{k\times k}-aI_{k})+D$ have same rank, inertia and signature.\\
Case I: Suppose $a>0$ and $-a+\dfrac{r_{k}}{n_{k}}\ge 0$. By Lemma \ref{l2}, we have $$\lambda_{k}(aJ_{k\times k}-aI_{k})+\lambda_{k}(D)\le \lambda_{k}(aJ_{k\times k}-aI_{k}+D).$$
Therefore, $$0\le-a+\dfrac{r_{k}}{n_{k}}\le\lambda_{k}(aJ_{k\times k}-aI_{k}+D).$$ Thus, $aJ_{k\times k}-aI_{k}+D$ is positive semidefinite. Since $M^{\prime\prime}$ and $aJ_{k\times k}-aI_{k}+D$ are congruent, it follows that $M^{\prime\prime}$ is positive semidefinite. Hence $S^{-}(M^{\prime\prime})=0.$ So, from (\ref{eq3}), we get
$$\mathcal{E}(M)=\sum_{i=1}^{k}\mathcal{E}(M_{i}).$$ 
Thus proof of (i) is done.\\
Case II: Suppose $M_{i}$ has at most one positive eigenvalue, namely $r_{i}$, for $1\le i\le k$, and $-a+\dfrac{r_{1}}{n_{1}}\le 0$. By Lemma \ref{l2}, we have $$\lambda_{2}(aJ_{k\times k}-aI_{k}+D)\le \lambda_{2}(aJ_{k\times k}-aI_{k})+\lambda_{1}(D)$$
Therefore, $$\lambda_{2}(aJ_{k\times k}-aI_{k}+D)\le -a+\dfrac{r_{1}}{n_{1}}\le 0.$$ 
Thus, $aJ_{k\times k}-aI_{k}+D$ has only one positive eigenvalue. Since $M^{\prime\prime}$ and $aJ_{k\times k}-aI_{k}+D$ are congruent matrices, it follows that $M^{\prime\prime}$ has only one positive eigenvalue. Also, since $M_{i}$ has at most one positive eigenvalue, namely $r_{i}$, for $1\le i\le k$, $M$ has only one positive eigenvalue by (\ref{spec}). Further since $trace(M)=0$, we must have $\mathcal{E}(M)=2\lambda_{1}(M)$. This proves (ii).\\
Case III: Suppose $a<0$ and $\dfrac{r_{1}}{n_{1}}+a(k-1)\le\lambda_{in_{i}}\le 0$ for all $1\le i\le k$. From Lemma \ref{l2}, we get
$$\lambda_{k-1}(aJ_{k\times k}-aI_{k})+\lambda_{k}(D)\le \lambda_{k-1}(aJ_{k\times k}-aI_{k}+D)$$
and
$$\lambda_{k}(aJ_{k\times k}-aI_{k}+D)\le \lambda_{k}(aJ_{k\times k}-aI_{k})+\lambda_{1}(D).$$
Therefore,
$$0\le -a+\dfrac{r_{k}}{n_{k}}\le \lambda_{k-1}(aJ_{k\times k}-aI_{k}+D)$$
and
$$\lambda_{k}(aJ_{k\times k}-aI_{k}+D)\le a(k-1)+\dfrac{r_{1}}{n_{1}}\le 0.$$
Thus $aJ_{k\times k}-aI_{k}+D$ has at most one negative eigenvalue. Since $aJ_{k\times k}-aI_{k}+D$ and $M^{\prime\prime}$  are congruent matrices, $M^{\prime\prime}$ has at most one negative eigenvalue. Moreover, $\lambda_{k}(M^{\prime\prime})\le \lambda_{k}(aJ_{k\times k}-aI_{k}+D)$, by Lemma \ref{ost}. Therefore  by (\ref{spec}), $\lambda_{k}(M^{\prime\prime})=\lambda_{n}(M)$ because $\lambda_{k}(M^{\prime\prime})\le a(k-1)+\dfrac{r_{1}}{n_{1}}\le \lambda_{in_{i}}$. Hence from equation (\ref{eq3}), we get $\mathcal{E}(M)=\sum_{i=1}^{k}\mathcal{E}(M_{i})-2\lambda_{n}(M)$. This completes the proof.
\end{proof}
\begin{cor}\rm{\cite{distance}}
The distance energy of  the complete multipartite graph $K_{n_{1},n_{2},\ldots,n_{k}}$ is equal to $4(n-k)$, where $n=n_{1}+n_{2}+\ldots+n_{k}$ and $n_{1}\ge n_{2}\ge \ldots\ge n_{k}\ge2$.
\end{cor}
\begin{proof}
From the definition of the distance matrix of a connected graph, the distance matrix of the complete multipartite graph $K_{n_{1},n_{2},\ldots n_{k}}$  is\\ $M[2(J_{n_{1}}-I_{n_{1}}),2(J_{n_{1}}-I_{n_{1}}),\ldots,2(J_{n_{1}}-I_{n_{1}}),1]$. Since $-1+\dfrac{r_{i}}{n_{i}}=-1+\dfrac{2(n_{i}-1)}{n_{i}}=1-\dfrac{2}{n_{i}}\ge 0$ for all $i$ and $\mathcal{E}(2(J_{n_{i}}-I_{n_{i}}))=4(n_{i}-1)$, from Theorem \ref{thm1} (i), we get $\mathcal{E}_{D}(G)=\sum_{i=1}^{n}\mathcal{E}(2(J_{n_{i}}-I_{n_{i}}))=\sum_{i=1}^{n}4(n_{i}-1)=4(n-k)$.
\end{proof}
\begin{cor}\label{cor}
	The reciprocal distance energy of  the complete multipartite graph $G=K_{n_{1},n_{2},\ldots, n_{k}}$ is equal to $2\lambda_{1}(RD(G))$.
\end{cor}
\begin{proof}
	From the definition of the reciprocal distance matrix of a connected graph, the reciprocal distance matrix of the complete multipartite graph $G$  is\\ $M[\frac{1}{2}(J_{n_{1}}-I_{n_{1}}),\frac{1}{2}(J_{n_{2}}-I_{n_{2}}),\ldots,\frac{1}{2}(J_{n_{k}}-I_{n_{k}}),1]$. Since $-1+\dfrac{r_{i}}{n_{i}}=-1+\dfrac{n_{i}-1}{2n_{i}}=-\dfrac{1}{2}-\dfrac{1}{2n_{i}}\le 0$ for all $i$, from Theorem \ref{thm1} (ii), we get $\mathcal{E}_{RD}(G)=2\lambda_{1}(RD(G)).$
\end{proof}
\begin{cor}\rm{\cite{seidel}}
The Seidel energy of the complete multipartite graph $G=K_{n_{1},n_{2},\ldots,n_{k}}$ with $k\ge3$ is equal to $2(n-k)-2\lambda_{n}(S(G))$, where $n=n_{1}+n_{2}+\ldots+n_{k}$.	
\end{cor}
\begin{proof}
From the definition of the Seidel matrix of a graph, the Seidel matrix of the complete multipartite graph $G$ is $S(G)=M[J_{n_{1}}-I_{n_{1}},J_{n_{2}}-I_{n_{2}},\ldots,J_{n_{k}}-I_{n_{k}},-1].$Therefore from Theorem \ref{thm1} (iii) we are done. 	
\end{proof}
\begin{cor}\rm{\cite{oenergy}}
The energy of the complete multipartite graph $G= K_{n_{1},n_{2},\ldots,n_{k}}$ is $2\lambda_{1}(A(G))$. 	
\end{cor}
\begin{proof}
The adjacency matrix of $G$ is $M=M[\textbf{0}_{n_{1}},\textbf{0}_{n_{2}},\ldots,\textbf{0}_{n_{k}},1]$, Where $\textbf{0}_{n_{i}}$ is a null matrix of order $n_{i}$. Therefore by Theorem \ref{thm1} (ii), the corollary follows. 
\end{proof}
\begin{cor}
The complementary distance energy of  $G=K_{n_{1},n_{2},\ldots,n_{k}}$ is $2\lambda_{1}(CD(G))$.
\end{cor}
\begin{proof}
	The complementary distance matrix \cite{cd} of the complete multipartite graph $G$  is $M[J_{n_{1}}-I_{n_{1}},J_{n_{2}}-I_{n_{2}},\ldots,J_{n_{k}}-I_{n_{k}},2]$. Since $-2+\dfrac{r_{i}}{n_{i}}=-2+\dfrac{n_{i}-1}{n_{i}}=-1+\dfrac{1}{n_{i}}\le 0$ for all $i$, from Theorem \ref{thm1} (ii), we get $\mathcal{E}_{CD}(G)=2\lambda_{1}(CD(G)).$
\end{proof}
\begin{cor}
	The reciprocal complementary distance energy of  $G=K_{n_{1},n_{2},\ldots,n_{k}}$ is $2(n-k)$, where $n=n_{1}+n_{2}+\ldots+n_{k}$ and $n_{1}\ge n_{2}\ge\ldots\ge n_{k}\ge2$.
\end{cor}
\begin{proof}
	The reciprocal complementary distance matrix \cite{cd} of the complete multipartite graph $G$  is $M[J_{n_{1}}-I_{n_{1}},J_{n_{2}}-I_{n_{2}},\ldots,J_{n_{k}}-I_{n_{k}},1/2]$. Since $-\dfrac{1}{2}+\dfrac{r_{i}}{n_{i}}=-\dfrac{1}{2}+\dfrac{n_{i}-1}{n_{i}}=\dfrac{1}{2}-\dfrac{1}{n_{i}}\ge 0$ for all $i$ and $\mathcal{E}(J_{n_{i}}-I_{n_{i}})=2(n_{i}-1)$, from Theorem \ref{thm1} (i), we get $\mathcal{E}_{RCD}(G)=2(n-k).$
\end{proof}
\section{Extremal complete multipartite graphs  with respect to reciprocal distance energy }
In this section, we show that among all complete $k$-partite graph on $n$ vertices, the complete split graph $CS(n,k-1)$ has minimum reciprocal distance energy and the reciprocal distance energy is maximum for the Turan graph $T(n,k)$.\\
We need the following Cauchy's interlace theorem.
\begin{lemma}\rm{\cite{book}}\label{bo}
Let $M$ be a Hermitian matrix of order $n$ and let $N$ be a principal submatrix of $M$ of order $k$. Then $\lambda_{n-k+i}(M)\le \lambda_{i}(N)\le \lambda_{i}(M)$ for $i=1,2,\ldots,k$.  	
\end{lemma}
\begin{lemma}\label{lemma5}
Let $G_{1}=K_{n_{1},n_{2},\ldots,n_{s},n_{s+1},\ldots, n_{k}}$ and $G_{2}=K_{n_{1},n_{2},\ldots,n_{s}-1,n_{s+1}+1,\ldots, n_{k}}$, where $n_{s}-n_{s+1}\ge 2$. Then $\lambda_{1}(RD(G_{2}))>\lambda_{1}(RD(G_{1}))$.	
\end{lemma}
\begin{proof}
Let  $V_{1}, V_2, V_{3},\ldots, V_{k}$ denote the vertex partition sets of the complete multipartite graph $G_{1}$. Let $V_{i}=\{v_{i1},v_{i2},\ldots, v_{in_{i}}\}$. Then form the definition of the complete multipartite graph, the vertices $v_{i1}, v_{i2},\ldots,v_{in_{i}}$ are symmetric. Since the reciprocal distance matrix of $G_{1}$ is non-negative, it follows from Perron-Frobenius theory that $\lambda_{1}(RD(G_{1}))$ is  simple and there exists an unit eigenvector $X$ (say) corresponding to the eigenvalue $\lambda_{1}(RD(G_{1}))$ with all its entries being positive. Let $x_{i1},x_{i2},\ldots x_{in_{i}}$ be the components of $X$ corresponding to the vertices $v_{i1}, v_{i2}, \ldots, v_{in_{i}}$. Suppose $x_{ij}\neq x_{ik}$ for $1\le j, k\le n_{i}$. Let $X^{\prime}$ be the vector obtained from $X$ by interchanging the components $x_{ij}$ and $x_{ik}$. Then the vectors $X$ and $X^{\prime}$ are linearly independent and $RD(G_{1})X^{\prime}=\lambda_{1}(RD(G_{1}))X^{\prime}$, because $x_{ij}\neq x_{ik}$ and the vertices $v_{i1}, v_{i2},\ldots,v_{in_{i}}$ are symmetric. Therefore multiplicity of $\lambda_{1}(RD(G_{1}))$ is at least 2, a contradiction. Hence $x_{i1}=x_{i2}=\ldots=x_{in_{i}}$ for all $1\le i\le k$. Let $n=n_{1}+n_{2}+\ldots+n_{k}$. From Rayleigh quotient inequality, we have
\begin{align}\label{ieq1}\lambda_{1}(RD(G_{2}))&\ge X^{T}RD(G_{2})X\notag\\
&=X^{T}\big[\dfrac{1}{2}(J_{n\times n}-I_{n})+\dfrac{1}{2}A(G_{2})\big]X\notag\\
&=X^{T}\dfrac{1}{2}(J_{n\times n}-I_{n})X+\dfrac{1}{2}X^{T}A(G_{2})X\notag\\
&=X^{T}\dfrac{1}{2}(J_{n\times n}-I_{n})X+\sum_{v_{ip}v_{jq}\in E(G_{2})}x_{ip}x_{jq}\notag\\
&=X^{T}\dfrac{1}{2}(J_{n\times n}-I_{n})X+\sum_{v_{ip}v_{jq}\in E(G_{1})}x_{ip}x_{jq}-\sum_{j=1}^{n_{s+1}}x_{sn_{s}}x_{(s+1)j}+\sum_{j=1}^{n_{s}-1}x_{sn_{s}}x_{sj}\notag\\
&=X^{T}\dfrac{1}{2}(J_{n\times n}-I_{n})X+\dfrac{1}{2}X^{T}A(G_{1})X-n_{s+1}x_{sn_s}x_{(s+1)1}+(n_{s}-1)x_{sn_s}x_{s1}\notag\\
&=X^{T}RD(G_{1})X-n_{s+1}x_{sn_s}x_{(s+1)1}+(n_{s}-1)x_{sn_s}x_{s1}\notag\\
&=\lambda_{1}(RD(G_{1}))-n_{s+1}x_{sn_s}x_{(s+1)1}+(n_{s}-1)x_{sn_s}x_{s1}.
\end{align}
Now, consider 
\begin{align*}\lambda_{1}(RD(G_{1}))x_{i1}&=\displaystyle\sum_{\underset{p\neq i}{p=1}}^{k}\sum_{q=1}^{n_{p}} x_{pq}+\dfrac{1}{2}\displaystyle\sum_{q=2}^{n_{i}}x_{iq}\\
&=\sum_{\underset{p\neq i}{p=1}}^{k}n_{p}x_{p1}+\dfrac{1}{2}(n_{i}-1)x_{i1}\\
&=\sum_{p=1}^{k}n_{p}x_{p1}-\dfrac{(n_{i}+1)}{2}x_{i1}.
\end{align*}
Therefore, 
\begin{equation}\label{ieq2}x_{i1}=\dfrac{2X}{2\lambda_{1}(RD(G_{1}))+n_{i}+1}, \text{where}~ X=\sum_{p=1}^{k}n_{p}x_{p1}.
\end{equation}
Using equation (\ref{ieq2}) in (\ref{ieq1}), we get $\lambda_{1}(RD(G_{2}))-\lambda_{1}(RD(G_{1}))$
\begin{eqnarray}\label{ieq3}
	 &\ge&2Xx_{sn_{s}}\left(\dfrac{-n_{s+1}}{2\lambda_{1}(RD(G_{1}))+n_{s+1}+1}+\dfrac{n_{s}-1}{2\lambda_{1}(RD(G_{1}))+n_{s}+1}\right)\notag\\[4mm]
&=&2Xx_{sn_{s}}\left(\dfrac{2\lambda_{1}(RD(G_{1}))(n_{s}-n_{s+1})-2\lambda_{1}(RD(G_{1}))+n_{s}-2n_{s+1}-1}{(2\lambda_{1}(RD(G_{1}))+n_{s+1}+1)(2\lambda_{1}(RD(G_{1}))+n_{s}+1)}\right)\notag\\[4mm]
&\ge&2Xx_{sn_{s}}\left(\dfrac{2\lambda_{1}(RD(G_{1}))+n_{s}-2n_{s+1}-1}{(2\lambda_{1}(RD(G_{1}))+n_{s+1}+1)(2\lambda_{1}(RD(G_{1}))+n_{s}+1)}\right)
\end{eqnarray}
Since $RD(K_{n_{s},n_{s+1}})$ is the principal submatrix of $RD(G_{1})$, by Lemma \ref{bo}, we have 
\begin{eqnarray}\label{ieq4}
\lambda_{1}(RD(G_{1}))&\ge& \lambda_{1}(RD(K_{n_{s},n_{s+1}}))\notag\\
&=&\dfrac{1}{4}\left(n_{s}+n_{s+1}-2+\sqrt{n_{s}^2+14n_{s}n_{s+1}+n_{s+1}^2}\right)\notag\\
&>&\dfrac{1}{2}(n_{s}+n_{s+1}-1)
\end{eqnarray}
Using equation (\ref{ieq4}) in (\ref{ieq3}), we get $\lambda_{1}(RD(G_{2}))-\lambda_{1}(RD(G_{1}))$
\begin{eqnarray*}
&>&2Xx_{sn_{s}}\left(\dfrac{2n_{s}-n_{s+1}-2}{(2\lambda_{1}(RD(G_{1}))+n_{s+1}+1)(2\lambda_{1}(RD(G_{1}))+n_{s}+1)}\right)\notag\\
&>&0.	
\end{eqnarray*} 
Thus $\lambda_{1}(RD(G_{2}))>\lambda_{1}(RD(G_{1}))$. Hence the proof of the theorem.
\end{proof}
\begin{theorem}
Let $G$ be a complete $k$-partite graph on $n$ vertices. Then 
$$\mathcal{E}_{RD}(CS(n,k-1))\le \mathcal{E}_{RD}(G)\le \mathcal{E}_{RD}(T(n,k)).$$
Moreover, the left equality holds if and only if $G\cong CS(n,k-1)$ and the right equality holds if and only if $G\cong T(n,k)$.
\end{theorem}
\begin{proof}
Let $G\cong K_{n_{1},n_{2},\ldots,n_{k}}$. Then $n=n_{1}+n_{2}+\ldots+n_{k}.$\\[2mm]	
Left inequality:   Suppose $n_{2}\ge2$. Then consider the graph $G_{1}\cong K_{n_{1}+1,n_{2}-1,\ldots,n_{k}}$. From Lemma \ref{lemma5}, we get 
$\lambda_{1}(RD(G))>\lambda_{1}(RD(G_{1}))$. Thus by Corollary \ref{cor}, we have $\mathcal{E}_{RD}(G)>\mathcal{E}_{RD}(G_{1})$. Now, if $n_{2}-1\ge 2$. Then 
by a similar argument, $\mathcal{E}_{RD}(G_{1})>\mathcal{E}_{RD}(G_{2})$, where $G_{2}\cong K_{n_{1}+2,n_{2}-2,\ldots,n_{k}}$. Thus $\mathcal{E}_{RD}(G)>\mathcal{E}_{RD}(G_{2})$. Repeating, $n_{2}-1$ times, we get $\mathcal{E}_{RD}(G)>\mathcal{E}_{RD}(K_{n_{1}+n_{2}-1,n_{3},\ldots,n_{k},1})$. Hence $\mathcal{E}_{RD}(G)>\mathcal{E}_{RD}(K_{n_{1}+n_{2}-1,n_{3},\ldots,n_{k},1})$ $>\mathcal{E}_{RD}(K_{n_{1}+n_{2}+n_{3}-2,n_{4},\ldots,n_{k},1,1})>\ldots>\mathcal{E}_{RD}(K_{n-k-1,1,\ldots,1,1})$. Therefore, $\mathcal{E}_{RD}(G)>CS(n,k-1)$.\\ [2mm]    
Right inequality: If $n_{s}-n_{s+1}\ge 2$ for some $1\le s< k$. Then by Lemma \ref{lemma5}, $\lambda_{1}(RD(G))<\lambda_{1}(RD(G_{1}))$, where $G_{1}\cong K_{n^{\prime}_{1},n^{\prime}_{2},\ldots,n_{s}^{\prime},n_{s+1}^{\prime},\ldots,n_{k}^{\prime}}$, $n_{i}^{\prime}=n_{i}$ for $i\neq s, s+1$, $n_{s}^{\prime}=n_{s}-1$ and $n_{s+1}^{\prime}=n_{s+1}+1$. Thus $\mathcal{E}_{RD}(G)<\mathcal{E}_{RD}(G_{1})$. Now, if $n^{\prime}_{t}-n^{\prime}_{t+1}\ge 2$ for some $1\le t< k$. Then by a similar argument, we get 
$\mathcal{E}_{RD}(G_{1})<\mathcal{E}_{RD}(G_{2})$, where $G_{2}\cong K_{n^{\prime\prime}_{1},n_{2}^{\prime\prime},\ldots,n_{k}^{\prime\prime}}$,    $n^{\prime\prime}_{i}=n^{\prime}_{i}$ for $i\neq t,t+1$, $n^{\prime\prime}_{t}=n^{\prime}_{t}-1$ and $n^{\prime\prime}_{t+1}=n^{\prime}_{t+1}+1$. So, $\mathcal{E}_{RD}(G)<\mathcal{E}_{RD}(G_{2})$.  Repeating the process until the difference of size of any two partitions is less than 2, we get $\mathcal{E}_{RD}(G)< \mathcal{E}_{RD}(T(n,k))$.
\end{proof}
\section{Reciprocal distance energy change of some complete multipartite graph due to edge deletion}
In this section, we study the change in reciprocal distance  energy of complete bipartite graph $K_{q,q}$ and the complete tripartite graph $K_{p,q,r}$ due to edge deletion.
\begin{lemma}{\rm\cite{book}}\label{cl}
Let $A$ and $B$ be two real symmetric matrices of same order such that $0\le A\le B$. Then $\lambda_{1}(A)\le \lambda_{1}(B)$. 
\end{lemma}
\ni
Let $E=\left[\begin{array}{cccc}
0&a&bJ_{1\times p}&cJ_{1\times q}\\
a&0&cJ_{1\times p }&bJ_{1\times q}\\
b1_{p}&c1_{n}&b(J_{p\times p}-I_{p})&cJ_{p\times q}\\
c1_{q}&b1_{q}&cJ_{q \times p}&b(J_{q\times q}-I_{q})			
\end{array}\right],$ where $a$, $b$ and $c$ are real constants. In the following lemma, we give the spectrum of the matrix $E$.
\begin{lemma}\label{ld1}
The spectrum of the matrix $E$ consists of 	$-b$ with multiplicity $p+q-2$ and the four roots of the polynomial $t^4+(-bp-bq+2b)t^3+(b^2pq-c^2pq-2b^2p-2b^2q-c^2p-c^2q-a^2+b^2)t^2+(2b^3pq-2bc^2pq+a^2bp+a^2bq-2abcp-2abcq-b^3p-b^3q-bc^2p-bc^2q-2a^2b)t-a^2b^2pq+a^2c^2pq+2pqab^2c-2pqac^3+b^4pq-2pqb^2c^2+c^4pq+a^2b^2p+qa^2b^2-2ab^2cp-2ab^2cq-a^2b^2$.
\end{lemma}
\begin{proof}
Let $e_{i,j}$ be a column vector of size $p+q+2$ with its $i$th and $j$th entries equal to 1 and -1, respectively, and the remaining entries are 0. Then 
$E e_{3,j}=-b e_{3,j}$ for $j=4,5,\ldots,p+2$ and $E e_{p+3,j}=-b e_{p+3,j}$ for $j=p+4,p+5,\ldots,p+q+2$. Thus $-b$ is an eigenvalue of $E$ corresponding to the $p+q-2$ linearly independent eigenvectors $e_{3,j}$ ($j=4,5,\ldots,p+2$) and $e_{p+3,j}$ ($j=p+4,p+5,\ldots,p+q+2$). Thus we have listed $p+q-2$ eigenvalues of $E$. Let $t_{1},t_{2},t_{3}$ and $t_{4}$ be the remaining eigenvalues of $E$ corresponding to the eigenvectors $X_{1}$, $X_{2}$, $X_{3}$ and $X_{4}$, respectively. Let $e_{i}$ be the column vector with its $i$th entry equal to 1 and the rest of the entries equal to 0. Also, let $e_{i}^{j}$ $(i\le j)$ be the column vector with its \textit{k}th entry equal to 1 if $i\le k\le j$, and 0 otherwise. Then the vectors $e_{1}$, $e_{2}$, $e_{3}^{p+2}$, $e_{p+3}^{p+q+2}$, $e_{3,j}$ ($j=4,5,\ldots,p+2$) and $e_{p+3,j}$ ($j=p+4,p+5,\ldots,p+q+2$) form a linearly independent set with $p+q+2$ elements. Since the matrix $E$ is real and symmetric, it has $p+q+2$ linearly independent eigenvectors. Thus the vectors $X_{1}$, $X_{2}$, $X_{3}$ and $X_{4}$ are in the linear span of the vectors $
e_{1}$, $e_{2}$, $e_{3}^{p+2}$ and $e_{p+3}^{p+q+2}$. Let $X_{i}=a_{i}e_{1}+b_{i}e_{2}+c_{i}e_{3}^{p+2}+d_{i}e_{p+3}^{p+q+2}$ for $i=1,2,3,4$. Then $EX_{i}=t_{i} X_{i}$ implies $b_{i}a+c_{i}bp+d_{i}cq=a_{i}t_{i}; a_{i}a+cc_{i}p+bd_{i}q=b_{i}t_{i}; ba_{i}+cb_{i}+b(p-1)c_{i}+cd_{i}q=c_{i}t_{i}; ca_{i}+bb_{i}+cc_{i}p+d_{i}b(q-1)=d_{i}t_{i}$. Thus  $EX_{i}=t_{i} X_{i}$ if and only if $$det\left(\begin{array}{cccc}
-t_{i}&a&bp&cq\\
a&-t_{i}&cp&bq\\
b&c&b(p-1)-t_{i}&cq\\
c&b&cp&b(q-1)-t_{i}
\end{array}\right)=0.$$
Therefore the remaining four eigenvalues of the matrix $E$ are the roots of the equation $t^4+(-bp-bq+2b)t^3+(b^2pq-c^2pq-2b^2p-2b^2q-c^2p-c^2q-a^2+b^2)t^2+(2b^3pq-2bc^2pq+a^2bp+a^2bq-2abcp-2abcq-b^3p-b^3q-bc^2p-bc^2q-2a^2b)t-a^2b^2pq+a^2c^2pq+2pqab^2c-2pqac^3+b^4pq-2pqb^2c^2+c^4pq+a^2b^2p+qa^2b^2-2ab^2cp-2ab^2cq-a^2b^2$.
 \end{proof}
\begin{lemma}\label{dl2}
The reciprocal distance spectrum of $K_{m,n}\backslash\{e\}$ consists of $-1/2$ with multiplicity $m+n-4$ and the four roots of the polynomial
$144\,{t}^{4}+ \left( -72\,m-72\,n+288 \right) {\lambda}^{3}+
\left(  \left( -108\,n-108 \right) m-108\,q+344 \right) {\lambda}^{2}
+ \left(  \left( -108\,n-22 \right) m-22\,n+136 \right) \lambda+
\left( 21\,n-41 \right) m-41\,n+57.$
\end{lemma}
\begin{proof}
We have the reciprocal distance matrix of $K_{m,n}\backslash\{e\}$ as \\
$$\left[\begin{array}{cccc}
	0&\frac{1}{3}&\frac{1}{2}J_{1\times p}&J_{1\times q}\\
	\frac{1}{3}&0&J_{1\times p }&\frac{1}{2}J_{1\times q}\\
	\frac{1}{2}1_{p}&1_{n}&\frac{1}{2}(J_{p\times p}-I_{p})&J_{p\times q}\\
	1_{q}&\frac{1}{2}1_{q}&J_{q \times p}&\frac{1}{2}(J_{q\times q}-I_{q})			
\end{array}\right],$$ where $p=m-1$ and $q=n-1$. Letting $a=1/3$, $b=1/2$ and $c=1$ in Lemma \ref{ld1}, we get the reciprocal distance spectrum of $K_{m,n}\backslash\{e\}$.  \end{proof}
\begin{theorem}
 We have $\mathcal{E}_{RD}(K_{q,q}\backslash \{e\})<\mathcal{E}_{RD}(K_{q,q})$ if $2\le q\le7 $, and	$\mathcal{E}_{RD}(K_{q,q}\backslash \{e\})>\mathcal{E}_{RD}(K_{q,q})$ if $q\ge8 $.
\end{theorem}
\begin{proof}
From Lemma \ref{dl2}, the reciprocal distance spectrum of $K_{q,q}\backslash \{e\}$ consists of $-1/2$ with multiplicity $2q-4$; 
$\frac{3}{4}\,q-\frac{5}{6}+\frac{1}{12}\,\sqrt {81\,{q}^{2}+72\,q-128}$; 
$\frac{3}{4}\,q-\frac{5}{6}-\frac{1}{12}\,\sqrt {81\,{q}^{2}+72\,q-128}$; $-\frac{1}{4}\,q-\frac{1}{6}+\frac{1}{12}\,\sqrt {9\,{q}^{2}+24\,q-32}$; $-\frac{1}{4}\,q-\frac{1}{6}-\frac{1}{12}\,\sqrt {9\,{q}^{2}+24\,q-32}$. Thus for $q=2$, $\mathcal{E}_{RD}(K_{q,q}\backslash \{e\})=\dfrac{4}{3}+\dfrac{1}{3}\sqrt{85}<\mathcal{E}_{RD}(K_{q,q})=5$, and for $q\ge3$, we have   $$\mathcal{E}_{RD}(K_{q,q}\backslash \{e\})=2\left(\frac{1}{2}\,q-1+\frac{1}{12}\,\sqrt {81\,{q}^{2}+72\,q-128}+\frac{1}{12}\,\sqrt {9\,{q}^{2}+24\,q-32}\right).$$
Let $X= \frac{1}{12}\,\sqrt {81\,{q}^{2}+72\,q-128}$ and $Y=\frac{1}{12}\,\sqrt {9\,{q}^{2}+24\,q-32}$. Then \begin{eqnarray*}
(2XY)^2-\left[\left(q+\frac{1}{2}\right)^2-(X^2+Y^2)\right]^2=\frac{1}{4}q^3-\frac{73}{48}q^2-\frac{35}{18}q-\frac{17}{16}.
\end{eqnarray*}
If $2\le q\le 7$. Then from the above equation, we get
\begin{eqnarray*}
(2XY)^2-\left[\left(q+\frac{1}{2}\right)^2-(X^2+Y^2)\right]^2<0.	
\end{eqnarray*}
Thus
\begin{eqnarray*}
	(X+Y)^2<\left(q+\frac{1}{2}\right)^2.	
\end{eqnarray*}
Therefore, $X+Y<q+\frac{1}{2}$ or  $\frac{1}{2}\,q-1+X+Y<\frac{3}{2}q-\frac{1}{2}$. Hence $\mathcal{E}_{RD}(K_{q,q}\backslash \{e\})<\mathcal{E}_{RD}(K_{q,q})$\\
If $q\ge 8$. Then we have 
\begin{eqnarray*}
(2XY)^2-\left[\left(q+\frac{1}{2}\right)^2-(X^2+Y^2)\right]^2&=&\frac{1}{4}q^3-\frac{73}{48}q^2-\frac{35}{18}q-\frac{17}{16}\\
&\ge&2q^2-\frac{73}{48}q^2-\frac{35}{18}q-\frac{17}{16}\\
&=&\frac{23}{48}q^2-\frac{35}{18}q-\frac{17}{16}\\
&\ge&\frac{17}{9}q-\frac{17}{16}\\
&>&0.
\end{eqnarray*}
\begin{eqnarray*}
	(X+Y)^2>\left(q+\frac{1}{2}\right)^2.	
\end{eqnarray*}
Therefore, $X+Y>q+\frac{1}{2}$ or  $\frac{1}{2}\,q-1+X+Y>\frac{3}{2}q-\frac{1}{2}$. Hence $\mathcal{E}_{RD}(K_{q,q}\backslash \{e\})>\mathcal{E}_{RD}(K_{q,q})$. This completes the proof.
\end{proof}
\begin{theorem}
We have $\mathcal{E}_{RD}(K_{p,q,r}\backslash \{e\})\le\mathcal{E}_{RD}(K_{p,q,r})$. 	
\end{theorem}
\begin{proof}
Case I: If $p,q\ge 2$. Then similar to Lemma \ref{ld1}, the reciprocal distance spectrum of $K_{p,q,r}\backslash \{e\}$ consists of $-1/2$ with multiplicity $p+q+r-5$, and the five roots of the polynomial $p(t)=32\,{t}^{5}+ \left( -16\,p-16\,q-16\,r+80 \right) {t}^{4}+\left(  \left( -24\,q-24\,r-32 \right) p+ \left( -24\,r-32 \right) q-
32\,r+104 \right) {t}^{3}+ \left[  \left(  \left( -20\,r-36
\right) q-36\,r-\right.\right.$\\$\left.\left.20 \right) p+ \left( -36\,r-20 \right) q-4\,r+68
\right] {t}^{2}+ \left[  \left(  \left( -20\,r-12 \right) q-12
\,r-10 \right) p+26+ \left( -12\,r\right.\right.$\\$\left.\left.-10 \right) q \right] t+
\left( -5\,r-3 \right) p+ \left( -5\,r-3 \right) q+3\,r+5$.\\
Case II: If p=q=1. Then the reciprocal distance spectrum  of $K_{p,q,r}\backslash \{e\}$ consists of $-1/2$ with multiplicity $r-1$, and the three roots of the polynomial $p(t)=8\,{\lambda}^{3}+ \left( 4-4\,r \right) {\lambda}^{2}+ \left( -16\,r-2
\right) \lambda-7\,r-1$.\\
Case III: If $p\ge2$ and $q=1$. Then the reciprocal distance spectrum  of $K_{p,q,r}\backslash \{e\}$ consists of $-1/2$ with multiplicity $p+r-3$, and the three roots of the polynomial $p(t)=16\,{t}^{4}+ \left( -8\,p-8\,r+24 \right) {t}^{3}+ \left( 
\left( -12\,r-24 \right) p-24\,r+24 \right) {t}^{2}+ \left( 
\left( -22\,r-16 \right) p\right.$\\$\left.-8\,r+12 \right) t+ \left( -5\,r-3
\right) p-2\,r+2.$\\
By Descarte's rule of signs, the polynomial $p(t)$ has exactly one positive root. Thus $\mathcal{E}_{RD}(K_{p,q,r}\backslash \{e\})=2\lambda_{1}(RD(K_{p,q,r}\backslash \{e\}))$. Since $\mathcal{E}_{RD}(K_{p,q,r})=2\lambda_{1}(RD(K_{p,q,r}))$ and $\lambda_{1}(RD(K_{p,q,r}\backslash \{e\}))\le \lambda_{1}(RD(K_{p,q,r}))$ by Lemma \ref{cl}, we are done. 	
\end{proof}

\end{document}